\documentclass[psamsfonts,12pt]{amsart}
\usepackage{graphicx}
\usepackage{amssymb,amsfonts}
\usepackage{amsmath, amscd}
\usepackage[all]{xy} 
\usepackage{enumerate}
\usepackage{mathrsfs}
\usepackage{amsthm}
\usepackage{eufrak}
\usepackage{fullpage} 
\usepackage{hyperref}

\newtheorem{thm}{Theorem}[section]

\theoremstyle{definition}

\newtheorem{exmp}[thm]{Example}

\theoremstyle{remark}
\newtheorem{rem}[thm]{Remark}

\makeatletter
\let\c@equation\c@thm
\makeatother
\numberwithin{equation}{section}

\title{Designer Ideals with High Castelnuovo-Mumford Regularity}

\author{Brooke Ullery}

\address{Department of Mathematics\\ University of Michigan\\ 530 Church Street,
Ann Arbor, MI  48109-1043}
\email{bullery@umich.edu}

\newcommand{\Sym}{{{\textrm{Sym}}}}

\newcommand{\reg}{{{\textrm{reg}\,}}}

\newcommand{\Tor}{{{\textrm{Tor}\,}}}

\newcommand{\C}{{{\mathbb{C}}}}

\renewcommand{\P}{{{\mathbb{P}}}}

\newcommand{\Z}{{{\mathbb{Z}}}}

\newcommand{\newword}[1]{\textbf{\emph{#1}}}

\begin{document}

\begin{abstract}
The purpose of this paper is to give a simple construction of ideals whose Castelnuovo-Mumford regularity is large compared to the generating degree.  Moreover, our ideals have the property that the Castelnuovo-Mumford regularity is revealed as late in the resolution as desired.
\end{abstract}

\maketitle 

\section{Introduction}

Let $S=\C[x_0, x_1, \ldots, x_n]$ and let $M$ be a finitely-generated graded $S$-module.  We define the \newword{Castelnuovo-Mumford regularity} of $M$, denoted $\reg(M)$, to be the smallest number $r$ so that each generator of the $i$th syzygy has degree at most $i+r$.  As regularity governs the complexity of modules, there has been considerable interest in bounding and estimating the regularity of modules, particularly that of ideals.  

For smooth projective varieties, the regularity of the defining ideal is bounded by the so called ``complete intersection regularity."  That is, for an ideal $I$ with generators of degree at most $k$ and codimension $N$, a result of Bertram, Ein, and Lazarsfeld \cite{BEL} shows that 
$$\reg{I} \leq N(k-1) + 1,$$
the expression on the right being the regularity of a complete intersection of degree $k$ hypersurfaces.  Chardin and Ulrich \cite{CU} gave a generalization of this result in the case of $I$ defining a subscheme with rational singularities.  However, for an arbitrary homogenous ideal, the situation is very different: the regularity can be doubly exponential in the number of variables, and polynomial in the generating degree.   Specifically, if $J \subseteq \C[t_0, \ldots, t_r]$ is an ideal generated in degree $k$, Bayer and Mumford \cite{BM}
observed that work of Giusti and Galligo leads to the bound
$$\reg(J) \leq (2k)^{2^{r-1}}.$$  
A proof of a characteristic-free bound of this form was given by Caviglia and Sbarra \cite{CS}.
It turns out that one can't do much better than this bound; in fact, Bayer and Stillman \cite{BS}
showed that a construction of Mayr and Meyer \cite{MM} proves the following fact:\footnote{Also see \cite{Y}, \cite{K}, and \cite{S} for more developments on this construction.}
\vspace{.3cm}

\begin{quote} There exists an ideal $I \subseteq \C[t_0, \ldots, t_r]$ generated in degree $k$ such that $\reg I \geq (k-2)^{2^{(r/10) -1}}$.
\end{quote}

\vspace{.3cm}
Caviglia \cite{C} gave a much simpler construction of ideals with lower, yet still polynomial, regularity that grows in the generating degree $k$ like $k^2$.   Chardin and Fall \cite {CF} gave examples of ideals in $r+3$ variables with regularity growing like $k^r$, where $r$ generators are of degree $k+1$, and two generators are of degree on the order of $2^r$.   Other than these constructions, there have been few examples of varieties with higher than complete intersection regularity, and the existing examples tend to be combinatorial and algebraic in nature.  In this paper, we give a simple construction that takes as input a module $M$ and outputs an ideal $J_M$ with similar regularity and other homological properties.  By applying our construction (detailed in \S 2) to pure modules, whose homological properties we can easily specify, we get a family of ideals whose regularity grows in the generating degree $k$ like $k^{(N-1)/(r-N)}$, where $N$ is the codimension of the zeroes of $I$. 

Furthermore, this family of ideals has the property that the regularity is ``revealed" late in the resolution.  In order to make this notion precise, we give a few more important definitions.  We define the \newword{graded Betti numbers} of $M$ to be 
$$\beta_{i,j}^S(M) = \dim(\Tor_i^S (M, \C))_j.$$
That is, if $F_\bullet \to M \to 0$ is a minimal graded free resolution, then $\beta_{i,j}^S(M)$ is the number of minimal generators of $F_i$ in degree $j$.  When the ring is clear from the context, we may write $\beta_{i,j}(M)$, or simply $\beta_{i,j}$. We arrange the Betti numbers in the \newword{Betti table} $\beta(M)$ as shown below.

\begin{center}$\beta(M) =$
\begin{tabular}{ r|ccccc}
&0&1&2&$\cdots$&$i$ \\ 
  \hline    
  0 & $\beta_{0,0}$ &$\beta_{1,1}$ & $\beta_{2,2}$ & $\cdots$ &\\                    
  1& $\beta_{0,1}$ & $\beta_{1,2}$ & $\beta_{2,3}$ & $\cdots$ &\\ 
  2&$\beta_{0,2}$ & $\beta_{1,3}$ & $\beta_{2,4}$ & $\cdots$ &\\ 
  $\vdots$& $\vdots$ & $\vdots$ & $\vdots$ & $\ddots$ &\\
  $j$& & & & & $\beta_{i, i+j}$ \\ 
\end{tabular}
\end{center}
\vspace{.3cm}
Notice that in terms of the Betti numbers of $M$ the Castelnuovo-Mumford regularity of $M$ is $$\reg{M} = \max\{j-i|\beta_{i,j}(M) \neq 0\}.$$  In other words, $\reg M$ measures the position of the lowest nonzero row of the graded Betti table.  Similarly, set $$t_i(M) =\max\{j|\beta_{i.j}(M)\neq0\}.$$  We call $(t_0(M), t_1(M), \ldots)$  the \newword{maximal degree sequence} of $M$.  (We define the \newword{minimal degree sequence} analogously, replacing ``max" with ``min.")  Note that $t_i(M)-i$ measures the position of the lowest nonzero entry in the $i$th column for each $i$.  In particular, if the degree sequence is strictly increasing (e.g. if $M$ is Cohen-Macaulay\footnote{For $M$ Cohen-Macaulay, the dual $F^*$ of the minimal free resolution is also a minimal free resolution.  The minimal degree sequence of $F^*$ is certainly increasing, as is the case for every resolution, and it corresponds to the negative of the maximal degree sequence of $M$. The claim follows immediately.}) and $t_r(M)$ is the last nonzero entry, then $$\reg(M) = t_r(M) - r.$$
We call $M$ a \newword{pure module} if for every $i$, $\beta_{i,j}\neq 0$ for at most one $j$.  That is, every module in its minimal free resolution is generated in a single degree.  In the case of pure modules, the maximal degree sequence gives all of the information about the degrees of the modules in the free resolution, so we call it simply the \newword{degree sequence}.

As Jason McCullough \cite{M} pointed out, the existing examples of ideals with high regularity
have regularity governed by the first syzygy.  That is, after $t_1(M)$, the maximal degree sequence increases by one at each step.  In this case, the Betti table has the following shape, where $k$ is the generating degree, and the support of the Betti table is contained in the shaded region.

\begin{center}
\includegraphics[width=110mm]{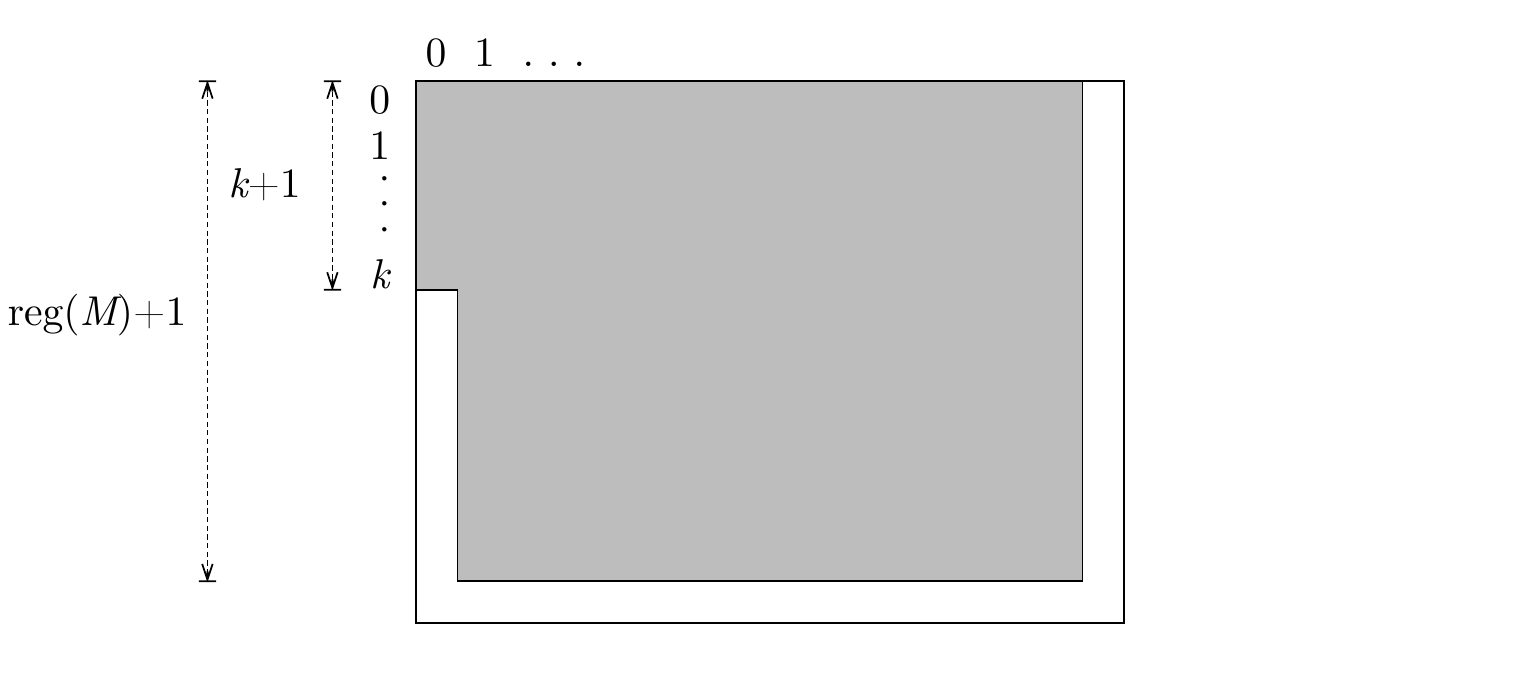}
\end{center}

By applying our construction to pure modules, we are able to find families of ideals that have high regularity (polynomial in the generating degree) that is revealed at the $n$th syzygy for any $n$.  The example below is also presented in full detail in \S 3.

\begin{exmp}

Let $S := \C[x_0, x_1, \ldots, x_n, y_1, y_2, \ldots, y_N]$ be the homogeneous coordinate ring for $\P^{n+N}$, and $X$ an $n$-plane given by the ideal $I  := (y_1, y_2, \ldots, y_N) \subseteq S$.  Let $M$ be a pure $S/I$-module with degree sequence $$(k, k+1, \ldots, k+n, k + n +1+d).$$  
Then the Betti table of $M$ has the following shape:

\begin{center}
\includegraphics[width=70mm]{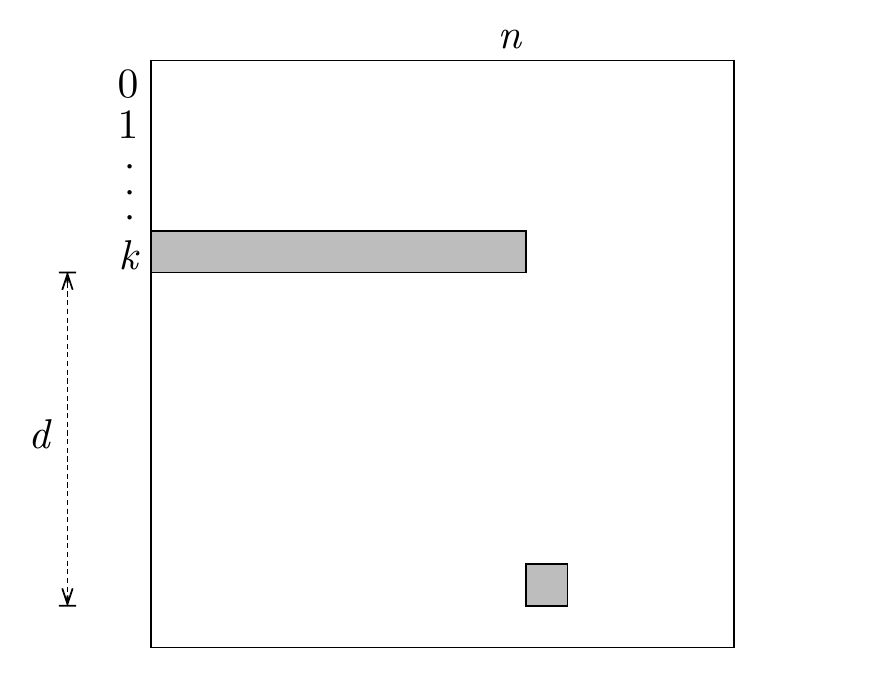}
\end{center}

Under certain restrictions on the parameters $N$, $k$, and $d$, we can construct an ideal $J_M \subseteq S$ with maximal degree sequence $$(k+1, k+2 \ldots, k+n, k + n +d+1, k+n+d+2, \ldots, k+n+d+N+1),$$ supported on $X$.  Given the restrictions, as $k$ increases, $n$ and $N$ held constant, the regularity of $J_M$ grows like $k^{(N-1)/n}$, so that the Betti table has the following shape:

\begin{center}
\includegraphics[width=130mm]{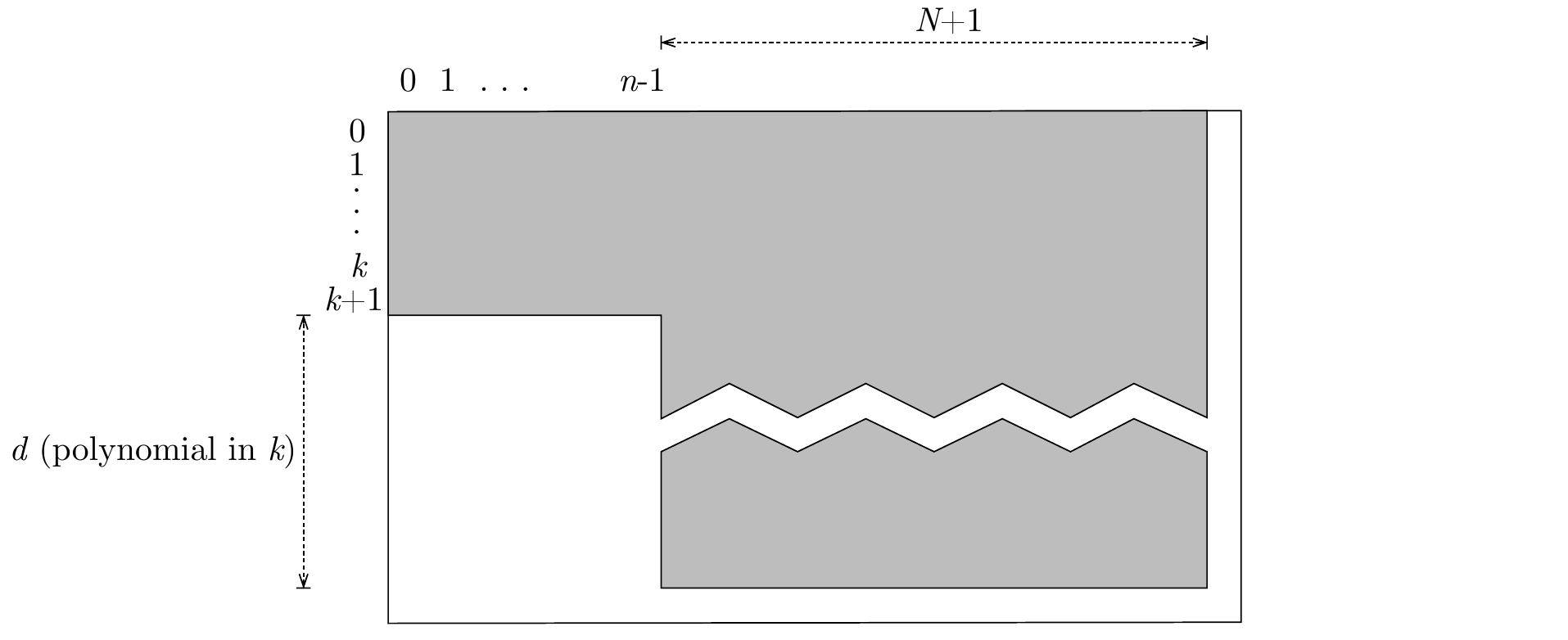}
\end{center}

\end{exmp}

In order to illustrate the properties of these ideals described above, we can look at the shape of the Betti table for specific large values of $k, n, N,$ and $d$.

\begin{exmp}
If we set

\begin{itemize}
\item generating degree $=k +1 = 51$, 
\item dimension of $X = n = 20$, and
\item codimension $=N = 5000$, 
\end{itemize}
then we can find an ideal with a Betti table having the following shape.  The dotted line bounds the support of the Betti table of a complete intersection generated in the same degree. (The figure is not drawn to scale.)

\begin{center}
\includegraphics[width=115mm]{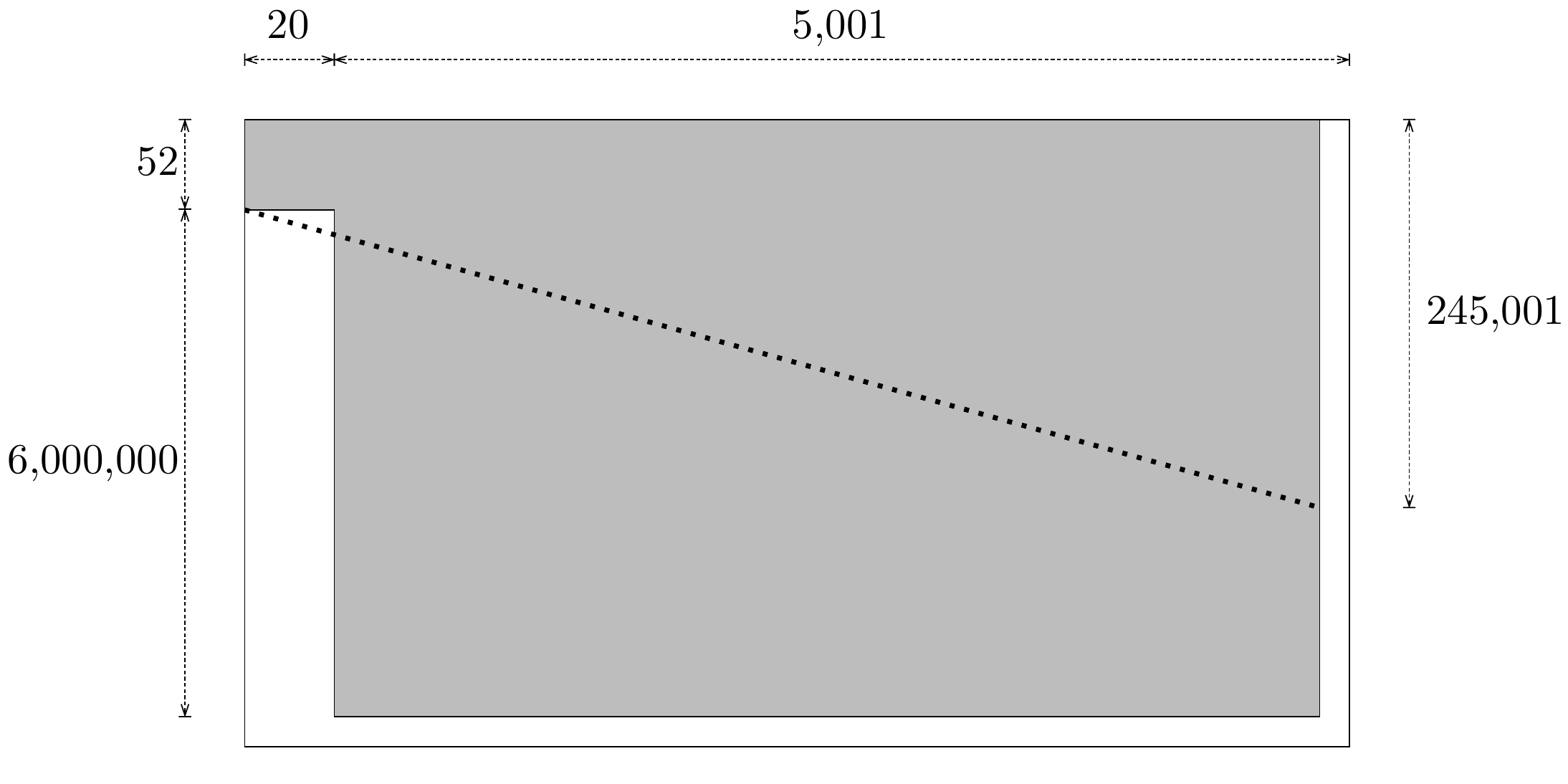}
\end{center}
\end{exmp}

\vspace{.5cm}
We can look at the above examples from a few different perspectives.  First of all, as stated earlier, this gives a method of constructing examples of ideals with high regularity governed as late in the resolution as desired.  In particular, the regularity is high even with respect to the degrees of the first syzygy, which answers a question raised by Huneke and McCullough (see the footnote on page 2 of \cite{DHS}) and casts doubt on the philosophy expressed in the same paper as well as in a paper of McCullough \cite{M} that the first syzygy contains most of the information about the regularity of an ideal.  Second, as we will illustrate in the next section, the construction is simple and algebraic, whereas earlier constructions are more combinatorial in nature.  

The examples above are special cases of a general family of ideals, described in our main theorem:

\begin{thm}
\label{degseq}
Let $X := \P^n$, and consider a linear embedding $X \subseteq \P^{n+N} =: Y$.  Define $$S := \C[x_0, x_1, \ldots, x_n, y_1, y_2, \ldots, y_N]$$ to be the homogeneous coordinate ring of $Y$, and $$I := (y_1, y_2, \ldots, y_N) \subseteq S$$ the homogenous ideal defining $X$.  Then $X$ has homogeneous coordinate ring $R := S/I$, which we can think of as $\C[x_0, x_1, \ldots, x_n]$.
Let $M$ be a finitely generated graded $R$-module, generated in positive degrees, with strictly increasing maximal degree sequence $(t_0, t_1, \ldots, t_r)$, where $r \leq n+1$.   Let $k$ be the minimum degree of a minimal generator of $M$.  Suppose
\begin{equation}\label{ineq} \sum_{j \in \Z} \beta_{0,j}^R(M) \leq \binom{k + N - 1}{k}.\end{equation}
Then there exists an ideal $J_M \subseteq S$ supported on $X$ such that $J_M$ has maximal degree sequence $$(t_1, t_2, \ldots, t_r, t_r +1, t_r +2, \ldots, t_r + N)$$ as an $S$-module.  In particular, $\reg J_M = \reg M + 1$.
\end{thm}

\section{The General Ideal Construction}

Before proving Theorem \ref{degseq}, we will give the construction of the ideal $J_M$.

Let $M$ be an $R$-module satisfying the hypotheses of the theorem, and $$\cdots \rightarrow F_2 \rightarrow F_1 \rightarrow F_0 \to M \to 0$$ a minimal graded $R$-free resolution of $M$.  Define $E := \ker (F_0 \to M)$, as below:

$$\xymatrix@=10pt{\cdots \ar[rr] && F_2 \ar[rr] && F_1 \ar[rr] \ar[dr] && F_0 \ar[rr] && M \ar[rr] && 0
\\ &&&&& E \ar[ur] \ar[dr] &&&&&
\\&&&& 0 \ar[ur] && 0 &&&&}$$
Since $k$ is the minimum degree of a minimal generator of $M$, we can write $$F_0 = \bigoplus_{j \geq k} R(-j)^{\oplus \beta_{i,j}}.$$
$I^k/I^{k+1}$ is the $k$th symmetric power of the conormal module of $X$ in $Y$ and we can write it as $$I^k/I^{k + 1}  \cong \Sym^k_R (I/I^2)  \cong R(-k)^{\oplus \binom{k + N - 1}{k}}. $$
Thus, since the rank of $F_0$ is less than the rank of $I^k/I^{k+1}$, by \eqref{ineq}, and since the degrees of the generators of $F_0$ are all at least $k$, there exist inclusions of $F_0$ into $I^k/I^{k+1}.$   Fix one such inclusion: $$F_0 \subseteq R(-k)^{\oplus \binom{k + N - 1}{k}} \cong I^k/I^{k + 1}.$$ 
Thus, $E \subseteq I^k/I^{k+1}$, so we get the following diagram:

$$\xymatrix{&&& E \ar@{^{(}->}[d] & \\ 0 \ar[r] & I^{k+1} \ar[r] & I^k \ar[r] & I^k/I^{k+1} \ar[r] & 0}.$$
Now, we define $J_M$ to be the lift of $E$ in $I^k$, so that we have an exact sequence
$$0 \to I^{k+1} \to J_{M} \to E \to 0.$$

The idea is that $I^{k+1}$ has a $(k+1)$-linear resolution, which allows $J_M$ to inherit its maximal degree sequence from $E$, and thus with our original module $M$.  The details are below.

\begin{proof}[Proof of Theorem \ref{degseq}]
By construction, $I^{k+1} \subseteq J_M \subseteq I^k$.  Since  $\sqrt{I^{k+1}} = \sqrt{I^k} = I$, we must have $\sqrt{J_M} = I$ as well.  Thus, $J_M$ is supported on $X$.

In order to find a resolution for $J_M$, we must first resolve $E$ and $I^{k+1}$.  

The minimal resolution for $E$ over $R$ is $\cdots G_2 \to G_1 \to G_0 \to E \to 0$, where $G_i := F_{i+1}$ for $i = 0, 1, \ldots$, and $F_\bullet$ the minimal resolution of $M$ as above.  Notice that $R[y_1, \ldots, y_N] \cong S$.  Thus, $G_\bullet [y_1, \ldots, y_N]$ is a minimal resolution for $E[y_1, \ldots, y_N]$ over $S$.  In order to recover $E$, we tensor with the Koszul complex with respect to the regular sequence $y_1, \ldots, y_N$, giving us the double complex $G_\bullet [y_1, \ldots, y_N] \otimes K_\bullet (y_1, \ldots, y_N)$, which is a minimal resolution of $E$ over $S$.  Denote this resolution $G'_\bullet$.

Our next step is to find the maximal degree sequence of $G'_\bullet$.  Notice that the maximal degree sequences of $G_\bullet [y_1, \ldots, y_N]$ and $K_\bullet (y_1, \ldots, y_N)$ are $(t_1, \ldots, t_r)$ and $(0, 1, \ldots, N)$, respectively.  Since $t_r > \cdots  > t_1 > t_0 > 0$, the maximal and minimal degree sequences of $G'_\bullet$ will be $(t_1, \ldots, t_r, t_r+1, t_r + 2, \ldots, t_r + N)$ and $(t_1, t_1+1, \ldots, t_1+N, t_2+N, \ldots, t_r+N)$, respectively.

Now we can construct the resolution of $J_M$.  The ideal $I^{k+1}$ has a pure linear resolution.  Applying the Horseshoe Lemma, we obtain the following commutative diagram

$$\xymatrix{ & \vdots \ar[d] &\vdots \ar[d]& \vdots \ar[d] & \\
0 \ar[r]& \oplus S(-(k+3)) \ar[d]\ar[r] &G''_2\ar[d] \ar[r]& G'_2 \ar[d] \ar[r]&0 \\ 
0 \ar[r]& \oplus S(-(k+2)) \ar[d] \ar[r]&G''_1 \ar[d] \ar[r]& G'_1 \ar[d]\ar[r]& 0\\ 
0 \ar[r]& \oplus S(-(k+1)) \ar[d] \ar[r] &G''_0 \ar[d] \ar[r]& G'_0 \ar[d]\ar[r]& 0 \\ 
0 \ar[r] & I^{k+1} \ar[d] \ar[r] & J_M \ar[r] \ar[d]& E \ar[r] \ar[d] & 0\\
&0 & 0& 0 &}$$
where the rows and columns are exact and $G''_i = \left( \oplus S(-(k+1+i))\right) \oplus G'_i$ for $0\leq i \leq N$, $G'_i$ otherwise.  To show that $G''_\bullet$ is minimal, it suffices to check that $G'_{i+1}$ does not have $S(-(k+1+i))$ as a summand for $i =0, \ldots, N$.  
The minimum degree of $G'_{i+1}$ is never less than $t_1+i+1$.  From our hypotheses, we have $t_1 > t_0 \geq k$, so $t_1+i + 1 >k+i+1$, as desired.  Thus, $J_M$ has maximal degree sequence $(t_1, \ldots, t_r, t_r+1, t_r + 2, \ldots, t_r + N)$.  In particular, $\reg J_M = t_r -r +1 = \reg M + 1$.

 \end{proof}
 
 \begin{rem}  In the proof above, we opted to write everything out using complexes, though it should be noted that a similar argument can be made using $\Tor $ modules.
 \end{rem}
 
 \section{A Family of Examples From Pure Modules}
 
 By applying our theorem to pure modules, we not only get a cleaner construction, but we also obtain a family of ideals with high asymptotic regularity.

With the same notation as in Theorem \ref{degseq}, let $M$ be a pure $R$-module with degree sequence $$(k, k+1, \ldots, k+n, k+n+1+d).$$  That is, its resolution is linear until the final syzygy, at which point there is a jump in degree of size $d+1$.  By the existence of pure modules, due to Eisenbud and Schreyer \cite{ES},
we can find such a resolution with $$\beta_0(M) = \binom{n+d}{n},$$ where $\beta_0(M)$ is the number of generators of $M$.  (In this specific case, the resolution happens to be the dual of the Eagon-Northcott complex, shifted by $k$.)

Thus, by the theorem, as long as we choose parameters $n$, $N$, $k$, and $d$ so that
$$\binom{n+d}{n} \leq \binom{k+N-1}{k},$$ we get an ideal $J_M$ supported on the smaller projective space $X$ with $$\reg J_M = \reg M + 1 = k + d +1.$$  Furthermore, the above inequality shows that $d$, and thus $\reg J_M$, grows in the generating degree $k$ like $ck^{(N-1)/n}$, where $c$ is a constant.

\section{Further Considerations}
The example in \S 3 was a very specific application of the general construction.  It may be interesting to apply the construction to other modules (possibly not pure) to see if the regularity would increase.  Furthermore, we found ideals supported on linear spaces due to the ease of calculating the conormal module and resolving modules over a polynomial ring.  It might be feasible to apply the construction to other projective varieties whose conormal modules are easily calculated -- possibly other complete intersections or Segre varieties.

\section*{Acknowledgements}
I would like to thank my advisor, Rob Lazarsfeld, for suggesting the problem and for his invaluable help along the way.  I would also like to thank both David Eisenbud and Jason McCullough for reading drafts and giving helpful suggestions.  Finally, thanks to Daniel Erman for the fruitful discussions.  

The work for this paper was partially supported by NSF RTG  grant DMS 0943832.

 \nocite{*}
 \bibliographystyle{mrl}
\bibliography{bibliography}

\begin{thebibliography}{10}

\bibitem{BM}
D.~Bayer and D.~Mumford, \emph{What can be computed in algebraic geometry?}, in
  Computational algebraic geometry and commutative algebra, 1--48, Cambridge
  Univ. Press (1993).

\bibitem{BS}
D.~Bayer and M.~Stillman, \emph{On the complexity of computing syzygies}, J.
  Symbolic Comput. \textbf{6} (1988), no. 2--3,  135--147.

\bibitem{BMNSSS}
J.~Beder, J.~Mc{C}ullough, L.~N\'u\~nez {B}etancourt, A.~Seceleanu, B.~Snapp,
  and B.~Stone, \emph{Ideals with larger projective dimension and regularity},
  J. Symbolic Comput. \textbf{46} (2011), no.~10,  1105--1113.

\bibitem{BEL}
A.~Bertram, L.~Ein, and R.~Lazarsfeld, \emph{Vanishing theorems, a theorem of
  {S}everi, and the equations defining projective varieties}, J. Amer. Math.
  Soc. \textbf{4} (1991), no.~3,  587--602.

\bibitem{BE}
D.~Buchsbaum and D.~Eisenbud, \emph{Remarks on ideals and resolutions},
  Symposia Mathematica \textbf{XI} (1971) 193--204.

\bibitem{C}
G.~Caviglia, Koszul algebras, Castelnuovo-Mumford regularity and generic
  initial ideals, Ph.D. thesis, University of Kansas (2004).

\bibitem{CS}
G.~Caviglia and E.~Sbarra, \emph{Characteristic-free bounds for the
  {C}astelnuovo-{M}umford regularity}, Compos. Math. \textbf{141} (2005),
  no.~6,  1365--1373.

\bibitem{CF}
M.~Chardin and A.~L. Fall, \emph{Sur la r\'egularit\'e de
  {C}astelnuovo-{M}umford des id\'eaux, en dimension 2}, C. R. Math. Acad. Sci.
  Paris \textbf{341} (2005), no.~4,  233--238.

\bibitem{CU}
M.~Chardin and B.~Ulrich, \emph{Liaison and {C}astelnuovo-{M}umford
  regularity}, Amer. J. Math. \textbf{124} (2002), no.~6,  1103--1124.

\bibitem{DHS}
H.~Dao, C.~Huneke, and J.~Schweig, \emph{Bounds on the regularity and
  projective dimension of ideals associated to graphs}, Journal of Algebraic
  Combinatorics  (2012) 1--19.

\bibitem{ES}
D.~Eisenbud and F.-O. Schreyer, \emph{Betti numbers of graded modules and
  cohomology of vector bundles}, J. Amer. Math. Soc. \textbf{22} (2009), no.~3,
   859--888.

\bibitem{K}
J.~Koh, \emph{Ideals generated by quadrics exhibiting double exponential
  degrees}, J. Algebra \textbf{200} (1998), no.~1,  225--245.

\bibitem{MM}
E.~Mayr and A.~Meyer, \emph{The complexity of the word problems for commutative
  semigroups and polynomial ideals}, Advances in Math \textbf{46} (1982)
  305--329.

\bibitem{M}
J.~Mc{C}ullough, \emph{A polynomial bound on the regularity of an ideal in
  terms of half of the syzygies}, Math. Res. Lett. \textbf{19} (2012), no.~3,
  555--565.

\bibitem{S}
I.~Swanson, \emph{The minimal components of the {M}ayr-{M}eyer ideals}, J.
  Algebra \textbf{200} (1998), no.~1,  225--245.

\bibitem{Y}
C.-K. Yap, \emph{A new lower bound construction for commutative {T}hue systems
  with applications}, J. Symbolic Comput. \textbf{12} (1991), no.~1,  1--27.

\end{thebibliography}

\end{document}